\documentclass{article} 

\usepackage[english]{babel}
\usepackage[utf8]{inputenc}
\usepackage[T1]{fontenc}

\usepackage[pdftex]{graphicx}
\usepackage{graphicx}
\usepackage{anysize}
\usepackage{euscript}
\usepackage{amscd}
\usepackage{marvosym}
\usepackage{upgreek}
\usepackage{color}
\usepackage{dsfont}
\usepackage{mathrsfs}
\usepackage{fancyhdr}

\usepackage{stmaryrd}

\usepackage[titletoc]{appendix}

\usepackage{amsmath, latexsym,amssymb, amsthm, amsfonts}

\newtheorem{theorem}{Theorem}

 \theoremstyle{plain}
  \newtheorem{thm}{Theorem}[section]
  \newtheorem{lem}[thm]{Lemma}
  \newtheorem{prop}[thm]{Proposition}
  \newtheorem{defn}[thm]{Definition}
  \newtheorem{cor}[thm]{Corollary}

\input{xy}
\xyoption{poly}
\xyoption{2cell}
\xyoption{all}

\def\modu{\mathop{\rm mod}\nolimits}

\def\Ker{\mathop{\rm Ker}\nolimits}

\newcommand{\pd}{\mathrm{pd}}
\newcommand{\id}{\mathrm{id}}
\newcommand{\phid}{\phi\mbox{-}\mathrm{dim}}
\newcommand{\psid}{\psi\mbox{-}\mathrm{dim}}

\newcommand{\E}{\mathrm {Ext^1_\Lambda}}

\newcommand{\Em}{\mathrm {Ext}^m_\Lambda}

\newcommand{\Ext}{\mathrm {Ext}}
\newcommand{\Ho}{\mathrm{Hom}_\Lambda}
\newcommand{\Hom}{\mbox{Hom}}

\newcommand{\mas}{$\underline{\mathrm {mod}}\Lambda$}
\newcommand{\chs}{$\underline{\chi}_X$}
\newcommand{\chss}{$\underline{\chi}_X$}
\newcommand{\inc}{\iota_{N_*}}

\newcommand{\pnc}{\pi_{N_*}}

\newcommand{\imct}{\iota_M^*}

\newcommand{\pmct}{\pi_M^*}
\newcommand{\On}{\mathbf{\Omega}}
\newcommand{\PM}{P_0(M)}
\newcommand{\PN}{P_0(N)}
\def\dar[#1]{\ar@<2pt>[#1]\ar@<-2pt>[#1]}

\newtheorem{rem}[thm]{Remark}

\newtheorem{eje}[thm]{Example}
\newtheorem{ejes}[thm]{Examples}

\newcommand{\N}{\mathbb{N}}

\title{Locally Gorenstein algebras,\\ a generalization of Gorenstein algebras}
\author{Leonardo Alarcon, María Andrea Gatica, Marcelo Lanzilotta}
\date{}

\begin{document}

\maketitle
\begin{abstract}
For an Artin algebra $\Lambda$, we define the class of locally Gorenstein modules, which is a generalization of semi-Gorenstein projective modules and, therefore, a generalization of projective Gorenstein modules. For a locally Gorenstein module $X\in\modu\Lambda$ we construct $\underline{{\chi}}_X$ a subcategory  of $\underline{{\rm mod}}$ $\Lambda$ and we prove that the syzygy functor is faithful and full on $\underline{{\chi}}_X$. Using that, we prove that the Igusa-Todorov functions $\phi$ and $\psi$ vanish on the family of locally Gorenstein modules. Finally, using locally Gorenstein modules, we define the locally Gorenstein algebras, a generalization of Gorenstein algebras, and we prove that their finitistic, and the $\phi$ and $\psi$ dimensions are finite.
\end{abstract}


\section{Introduction}
In this work $\Lambda$ will denote an Artin algebra over a commutative artinian ring with unity $R$, where $\modu\Lambda$ will denote  the category of finitely generated left modules. In addition $\underline{\modu}\Lambda$ will denote the quotient of $\modu\Lambda$ by the ideal of morphisms which factorize by projectives modules in $\modu\Lambda$.
 
The concept of Gorenstein projective module (see 2.1) goes back to a work of Auslander and Bridger in 1969 \cite{AB}, where the $G$-dimension of a finitely generated module over a Noetherian ring has been introduced. Enochs and Jenda in 1995 \cite{EJ} defined the concept of Gorenstein projective module, and now a day is clear that a module $M$ is Gorenstein projective if and only if the $G$-dimension of $M$ is zero. 

Among the many properties that Gorenstein projective modules satisfy, we highlight the following two: on the one hand, Gorenstein projective modules are semi-Gorenstein projective modules, that is, they belong to $^\perp\Lambda$ (see 2.2) and, on the other hand, they vanish for the Igusa-Todorov functions (see 2.2 and \cite{LM}). 

In this work we introduce a new concept \textit{locally Gorenstein projective modules} (see Definition 3.1) that generalize semi-Gorenstein projective modules (and therefore also Gorenstein projective modules) 
but, nevertheless, the Igusa - Todorov functions still have a value of zero on this family of modules (see Theorem 2).

It is worth noting that to prove Theorem 2 we use a construction similar to that used by Auslander and Bridge in \cite{AB}, namely, starting from a locally Gorenstein projective module we construct a full subcategory of $\underline{\modu}\Lambda$, which we denote by $\underline{\chi}_X$, and we prove that the functor $\Omega$ restricted to  $\underline{\chi}_X$ is faithful and full (see Theorem 1).

From the locally Gorenstein projective modules, we define the \textit{locally Gorenstein algebras} (see Definition \ref{def algebras l-G}) by generalizing the concept of Gorenstein algebras defined by Auslander and Reiten (see \cite{AR}). Gorenstein algebras have very nice homological behaviors, their homological dimensions: finitistic dimension, $\phi$-dimension and $\psi$-dimension are finite, and therefore satisfy the finitistic Conjecture. These properties are still preserved in the generalization to locally Gorenstein algebras (see Proposition 5.4 and Theorem 3). It is worth mentioning at this point that there are at least, two families of algebras which also generalize Gorenstein algebras. We are referring to virtually Gorenstein algebras, defined by Beligiannis and Reiten in \cite{BelR}; and weakly Gorenstein algebras, defined by Ringel and Zhang in \cite{RZ}. In a recent paper, Li is said to have shown that every virtually Gorenstein algebra is weakly Gorenstein (see \cite{Li}). The new family of algebras that we define, locally Gorenstein algebras,  does not coincide with any of these families. For example the bound quiver algebra given by the quiver $\xymatrix@C=5mm{&1\ar@(ul,dl)_{\alpha}\ar[r]^{\beta}&2}$ and the relations $\{\alpha^2, \beta\alpha\}$ is virtually Gorenstein and also weakly Gorenstein but is not locally Gorenstein.

Finally, in \ref{seccion 5.2}, we study locally Gorenstein algebras arising from a Morita context (see \cite{AMIT}, \cite{GP}). We establish sufficient conditions for a Morita context algebra to be locally Gorenstein, thus obtaining a considerable family of examples of such type of algebras.

\section{Preliminaries}

\subsection{$\Lambda$-Gproj and $^\perp\Lambda$}

In this section, we recall the definition of Gorenstein algebra, as well as the definition of projective Gorenstein modules over an Artin algebra.\\

An Artin algebra $\Lambda$ is a Gorenstein algebra if the id$(_\Lambda\Lambda)<\infty$ and the id$(\Lambda_\Lambda)<\infty$. In addition, we say that $\Lambda$ is $n$-Gorenstein if $n$ is the smallest integer such that id$(_\Lambda\Lambda)\leq n$ and id$(\Lambda_\Lambda)\leq n$.\\

A complex of projective modules is totally acyclic if it is an exact sequence of projective $\Lambda$-modules:
$$\xymatrix{(\mathcal{P},p): &&\cdots\ar[r]^{p_{-2}}&P_{-1}\ar[r]^{p_{-1}}&P_0\ar[r]^{p_0}&P_1\ar[r]^{p_1}&\cdots}$$

\noindent  and the complex:
$$\xymatrix{\rm{Hom}_\Lambda((\mathcal{P},p),\Lambda):  & \cdots\ar[r]^{p_{1}^*\hspace{5mm}}&{\rm Hom}_\Lambda(P_{1},\Lambda) \ar[r]^{p_{0}^*}&{\rm Hom}_\Lambda(P_0,\Lambda)\ar[r]^{p_{-1}^*}&{\rm Hom}_\Lambda(P_{-1},\Lambda)\ar[r]^{\hspace{8mm}p_{-2}^*}&\cdots}$$
\noindent is also exact.\\

A finitely generated $\Lambda$-module $G$ is said to be Gorenstein projective, provided that there is a totally acyclic complex $(\mathcal{P},p)$ of projective modules such that $G\cong{\rm Ker}(p_0)$.\\

In particular, for all $i\in\mathbb{Z}$, the $\Lambda$-module Ker$(p_i)$ is Gorenstein projective. We will denote by $\Lambda$-Gproj the full subcategory of mod $\Lambda$ determined by the Gorenstein projective $\Lambda$-modules.

We will denote with $^\perp\Lambda$ the subfamily of the semi-Gorenstein projective modules of  $\modu \Lambda$ (see \cite{RZ}), that is, the subfamily defined by
$^\perp\Lambda=\{X\in \modu\,\Lambda:{\rm Ext}^i_\Lambda(X,\Lambda)=0,\, \forall i\geq 1\}.$

\begin{rem}\label{gproj son estables}
\begin{enumerate}

\item Let $X$ be a $\Lambda$-module and let
$$\xymatrix{\cdots\ar[r]^{q_2}&Q_1\ar[r]^{q_1}&Q_0\ar[r]^{q_0}&X\ar[r]&0}$$
 \noindent be a projective resolution of $X$. Then, $X\in\ \!\!\! ^\perp\Lambda$ if an only if the sequence 
$$(*)\ \xymatrix{0\ar[r]&{\rm Hom}_\Lambda(X,\Lambda)\ar[r]^{q_0^*}&{\rm Hom}_\Lambda(Q_0,\Lambda)\ar[r]^{q_1^*}&{\rm Hom}_\Lambda(Q_1,\Lambda)\ar[r]^{\hspace{8mm}q_2^*}&\cdots}$$
\noindent is exact. In fact, the sequence $(*)$ is exact if an only if ${\rm Ext}^i_\Lambda(X,\Lambda)=0$ for all $i\geq 1$, thus the sequence $(*)$  is  exact if an only if $X\in\ \!\! ^\perp\Lambda$.

\item All Gorenstein projective $\Lambda$-modules belong to $^\perp\Lambda$, that is, $\Lambda-{\rm Gproj}\subseteq ^\perp\!\negmedspace\Lambda$. In fact, let $M$ be a Gorenstein projective $\Lambda$-module. By definition, there is an acyclic complex, $(\mathcal{P},p)$, of projective modules:
$$\xymatrix{(\mathcal{P},p)=&&\cdots\ar[r]^{p_{-2}}&P_{-1}\ar[r]^{p_{-1}}&P_0\ar[r]^{p_0}&P_1\ar[r]^{p_1}&\cdots}$$
with $M\cong {\rm Ker}(p_0)$, such that the following sequence:  
$$\xymatrix{\cdots\ar[r]^{p_{1}^*\hspace{8mm}}&{\rm Hom}_\Lambda(P_{1},\Lambda) \ar[r]^{p_{0}^*}&{\rm Hom}_\Lambda(P_0,\Lambda)\ar[r]^{p_{-1}^*}&{\rm Hom}_\Lambda(P_{-1},\Lambda)\ar[r]^{\hspace{8mm}p_{-2}^*}&\cdots}$$
is exact.

\noindent Then  ${\rm Ext}^i_\Lambda(M,\Lambda)={\Ker}(p_{-(i+1)}^*)/{\rm Im}(p_{-i}^*)=0$ for all $i\geq 1$.
\end{enumerate}
\end{rem}

The following provides a characterisation of $n$-Gorenstein algebras.

\begin{thm}\cite{C}\label{teo de Chen}
Let $\Lambda$ be an Artin algebra and let $n\geq 0$. Then,  $\Lambda$ is  $n$-Gorenstein if an only if $\Lambda$-Gproj$=\Omega^n({\rm mod}\,\Lambda)$. 
 
\end{thm}

\begin{rem}\label{obs ortogonal igual a Gproj} If $\Lambda$ is a $n$-Gorenstein algebra then  $\Lambda$-Gproj$=\,^\perp\Lambda$ \cite[Theorem 2.3.3]{C}.
\end{rem}

\subsection{Igusa-Todorov functions}
In this section we recall the definitions and some properties of the functions defined by K. Igusa and G. Todorov in \cite{IT}.

Let $K_0$ be the quotient of the free abelian group generated by the symbols of the form $[M]$, where $M\in \modu\Lambda$, by the subgroup generated by
\begin{enumerate}
\item $[A]-[B]-[C]$ if $A\cong C\oplus B$, with $A,\,B,\,C$ in $\modu \Lambda$.
\item $[P]$ if $P$ is a projective $\Lambda$-module.
\end{enumerate}

Note that $K_0$ is the free abelian group generated by the isomorphic classes of non-projective indecomposable $\Lambda$-modules.
Then syzygy $\Omega$ defines the morphism of groups $\mathbf{\Omega}: K_0\rightarrow K_0$ given by $\mathbf{\Omega}([M])=[\Omega(M)]$ for all  $[M]$ in $K_0$.  

For each free group $X$, rg$\,X$ denotes the rank of $X$. Given a finitely generated $\Lambda$ module $M$, we denote by $\langle {\rm add}\, M\rangle$ the subgroup of $K_0$ generated by the indecomposable direct summands of $M$.

We recall that \textbf{the first Igusa-Todorov function or function $\phi$ of Igusa-Todorov} of $M\in\modu\Lambda$ is defined as the smallest non-negative integer, $\phi(M)$, such that:  $${\rm rg}\,\mathbf{\Omega}^{\phi(M)}(\langle {\rm add}\, M\rangle)={\rm rg}\,\mathbf{\Omega}^{\phi(M)+s}(\langle {\rm add}\, M\rangle),$$ or equivalently, the smallest non-negative integer such that:
$$\mathbf{\Omega}:\mathbf{\Omega}^{\phi(M)+s}(\langle{\rm add}\,M\rangle)\rightarrow \mathbf{\Omega}(\mathbf{\Omega}^{\phi(M)+s}(\langle{\rm add}\,M\rangle))$$ is a monomorphism for all $s\in \N$.

\begin{prop}\cite{IT}\label{propiedades de phi} Let $M$ and $N$ be modules in $\modu \Lambda$, then the following statements hold:

\begin{enumerate}
\item If $\pd (M)<\infty$ then $\phi(M)=\pd (M)$.
\item If $M$ is indecomposable with $\pd(M)=\infty$, then $\phi(M)=0$.
\item $\phi(M)\leq \phi(M\oplus N)$.
\item $\phi(sM)=\phi(M)$ for all $s\in \mathbb{N}$, where $sM$ denote the direct sum of $M$ $s$-times.

\end{enumerate}
\end{prop}
The \textbf{second Igusa-Todorov function or function $\psi$ of Igusa-Todorov} of $M\in\modu\Lambda$ is defined as $$\psi(M)=\phi(M)+{\rm max}\{\pd(N):\,N| \Omega^{\phi(M)}(M){\rm\, \ with\ \,} \pd(N)<\infty \},$$
where $N|L$ denotes that $N$ is a direct summand of $L$.

\begin{prop}\cite{IT}\label{propiedades de la psi} Let $M,N$ be in $\modu\Lambda$. Then:
\begin{enumerate}
\item If $\pd (M)<\infty$ then $\psi(M)=\pd (M)$.
\item $\psi(M)\leq \psi(M\oplus N)$.
\item $\psi(sM)=\psi(M)$ for all $s\in \mathbb{N}$.
\item If $N|\Omega^nM$ such that $n\leq \phi(M)$ and $\pd(N)<\infty$, then $\pd N+n\leq \psi(M)$.
\end{enumerate}
\end{prop}

\begin{prop}\cite{HLM}\label{phi meno que phi de omega} For each module $M$ in $\modu\Lambda$, the following inequalities hold:
\begin{enumerate}
\item $\phi(M)\leq \phi(\Omega(M))+1$.
\item $\psi(M)\leq \psi(\Omega(M))+1$.
\end{enumerate}
\end{prop}

Finally, we recall that the $\phi$-dimension of $\Lambda$ and $\psi$-dimension of $\Lambda$ are defined as
$\phid (\Lambda)={\rm sup}\{\phi(M): M\in\modu\Lambda\}$ and $\psid (\Lambda)={\rm sup}\{\psi(M): M\in\modu\Lambda\}$, respectively.

\begin{prop}\label{prop desiguladades entre dim} Let $\Lambda$ be an Artin algebra.Then, the following inequalities hold:
$${\rm fin.dim}(\Lambda)\leq\phid(\Lambda)\leq \psid(\Lambda)\leq {\rm gl.dim}(\Lambda).$$
where fin.dim$(\Lambda)$ and gl.dim$( \Lambda)$ denotes the finitistic dimension and global dimension of $\Lambda$, respectively.
 
\end{prop}
\subsection{Morita context algebras}
Let $A$ and $B$ be two Artin algebras, a Morita context is a 6-tuple $\mathcal{M}=(A,M,N,B,\alpha, \beta)$ where  $_BM_A$ is a $B$-$A$-bimodule, $_AN_B$ is an $A$-$B$-bimodule and the maps $\beta: M\otimes_AN\rightarrow B$, $\alpha: N\otimes_BM\rightarrow A$ are $B$-$B$ and \linebreak $A$-$A$-bimodule homomorphisms respectively, satisfying the following associativity conditions, $\forall m,m' \in M,$ $\forall n,n'\in N$:
$$\beta(m\otimes n)m'=m\alpha(n\otimes m') \mbox{ and } n\beta(m\otimes n')=\alpha(n\otimes m)n'.$$

From the Morita context $\mathcal{M}=(A,M,N,B,\alpha,\beta)$ we define the Morita ring
 $$\Lambda_{(\alpha,\beta)}=\begin{pmatrix} A & _AN_B \\
_BM_A & B\end{pmatrix}$$
where the addition is the usual one  and the multiplication is given by
$$\begin{pmatrix}a&n\\m&b\end{pmatrix}.\begin{pmatrix}a'&n'\\m'&b'\end{pmatrix}=\begin{pmatrix}aa'+\alpha(n\otimes m')&an'+nb'\\ ma'+bm'&bb'+\beta(m\otimes n')\end{pmatrix}.$$

Now we recall the description of modules on Morita rings $\Lambda_{(\alpha, \beta)}$. Any left $\Lambda_{(\alpha, \beta)}$-module can be represented by a 4-tupla $(X,Y,f,g)$ where $X\in \modu A$, $Y\in \modu B$, $f\in \rm{Hom}_B(M\otimes_A X,Y)$ and $g\in \rm{Hom}_A(N\otimes_B Y,X)$ such that the following
diagrams commute: 
$$\xymatrix@=13mm{N\otimes_BM\otimes_AX\ar[d]_{\alpha\otimes Id_X}\ar[r]^{\hspace{0.2 cm}Id_N\otimes f}&N\otimes_B Y\ar[d]^g\\ A\otimes_A X\ar[r]^{\simeq}&X}\hspace{.5 cm}\xymatrix@=13mm{M\otimes_AN\otimes_BY\ar[d]_{\beta\otimes Id_Y}\ar[r]^{\hspace{0.2 cm}Id_M\otimes g}&M\otimes_A X\ar[d]^f\\ B\otimes_B Y\ar[r]^{\simeq}&Y}.$$
We denote $\alpha_X$ and $\beta_Y$ the compositions of the following morphisms:
$$\alpha_X:\xymatrix@=13mm{N\otimes_BM\otimes_AX\ar[r]^{\hspace{0.2 cm}\alpha\otimes Id_X}&A\otimes_A X\ar[r]^{\simeq}&X},\ \beta_Y:\xymatrix@=13mm{M\otimes_AN\otimes_BY\ar[r]^{\hspace{0.2 cm}\beta\otimes Id_Y}&B\otimes_B Y\ar[r]^{\simeq}&Y.}$$

A  $\Lambda_{(\alpha,\beta)}$-morphism $(X,Y,f,g)\rightarrow (X',Y',f',g')$ is a pair $(a,b)$ of morphisms, with $a\in \rm{Hom}_A(X,X')$ and $b\in \rm{Hom}_B(Y,Y')$ such that the following diagrams are commutative:
$$\xymatrix@=13mm{M\otimes_A X\ar[r]^f\ar[d]_{Id_M\otimes a}&Y\ar[d]^b\\M\otimes_A X'\ar[r]^{f'}&Y'} \hspace{.5 cm} \xymatrix@=13mm{N\otimes_B Y\ar[r]^g\ar[d]_{Id_N\otimes b}&X\ar[d]^a\\N\otimes_B Y'\ar[r]^{g'}&X'}.$$ 

Let us observe that a $\Lambda_{(\alpha,\beta)}$-module is an indecomposable projective $\Lambda_{(\alpha,\beta)}$-module if and only if it has the form $(P,M\otimes_AP,Id_{M\otimes P}, \alpha_P)$ or the form $(N\otimes_BQ,Q, \beta_Q,Id_{N\otimes Q})$ where $P$ and $Q$ are indecomposable projective $A$-module and $B$-module, respectively. Observe that the  $\Lambda_{(0,0)}$ indecomposable projective modules are represented as  $(P,M\otimes_AP,Id_{M\otimes P}, 0)$ or $(N\otimes_BQ,Q, 0,Id_{N\otimes Q})$ with $P$ and $Q$ as before.

Finally, let us recall the following functors:

\begin{enumerate}
    \item The functor $T_A:Mod A\rightarrow Mod\Lambda$ is defined by $T_A(X)=(X,M\otimes_A X,Id_{M\otimes_A X}, \alpha_X )$ on the objects and if $a:X\rightarrow X'$ is a $A$ -morphism then $T_A(a)=(a,M\otimes a)$.
    \item The functor $U_A:Mod \Lambda\rightarrow ModA$ is defined by $U_A((X,Y,f,g))=X$  on the objects and if       \linebreak    $(a,b):(X,Y,f,g)\rightarrow (X',Y',f',g')$ is a $\Lambda$-morphism then $U_A((a,b))=a$.
    \item The functor  $T_B:Mod B\rightarrow Mod\Lambda$ is defined by $T_B(Y)=(N\otimes_B Y,Y,\beta_Y,Id_{N\otimes_B Y})$  on the objects and if $b:Y\rightarrow Y'$ is a $B$ -morphism then  $T_B(b)=(N\otimes b,b)$.

    \item The functor $U_B:Mod \Lambda\rightarrow ModB$ is defined by $U_B((X,Y,f,g))=Y$  on the objects and if       \linebreak    $(a,b):(X,Y,f,g)\rightarrow (X',Y',f',g')$ is a $\Lambda$-morphism then $U_B((a,b))=b$.
 
\end{enumerate}
\begin{prop}\cite{GP} The pairs $(T_A,U_A)$ and $(T_B,U_B)$ are adjoint pairs of functors.
\end{prop}

\section{Locally Gorenstein modules: definitions and basic properties}
\subsection{Locally Gorenstein modules: definitions and basic properties}
\begin{defn}\label{1}
Let $X$ be a finitely generated $\Lambda$-module and $\mathbf{P}_{\Lambda} X=\{P_1,...,P_r\}$ the set of all indecomposable projective modules that are direct summands of projective modules at the minimal projective resolution of $X$. We say that $X$ is a {\bf locally Gorenstein module} if $\Em (X,P_j)=0$ for all $P_j\in \mathbf{P}_{\Lambda} X$ and for all $m\geq 1$. 
\end{defn}

\begin{ejes}\label{2}

\begin{enumerate}
\item
Let $\Lambda$ be an Artin algebra. Then:
\begin{enumerate}
\item All projective modules are locally Gorenstein modules. 

\item Let $\mathcal{P}$ be the class of all projective $\Lambda$-modules with finite injective dimension, \linebreak$n=\max_{P_i\in\mathcal{P}}\{\id (P_i)\}$ and $X$ a $\Lambda$-module with infinite projective dimension such that the class $\mathbf{P}_{\Lambda} \Omega^n(X)\subseteq \mathcal{P}$. Then $\Omega^m(X)$ is a locally Gorenstein module, for all $m\geq n$.

\end{enumerate}

\item Consider the following radical square zero algebra $\Lambda=kQ/F^2$ where $Q$ is defined as follows:
$$\xymatrix{1\ar[r]&2\ar[r]&3\ar[r]&4\ar@(ur,dr)}$$

The simple $S_4$ is a locally Gorenstein module.

\end{enumerate}
\end{ejes}

\begin{prop}
Let $\Lambda$ be an Artin algebra and $M$ a $\Lambda$-module. Then, $M$ is locally Gorenstein if and only if,
Hom$_\Lambda(-,P_j)$, with $P_j$ in $\mathbf{P}_{\Lambda} M$, is exact on any projective resolution of $M$. 

\end{prop}

\begin{cor} Let $\Lambda$ be an Artin algebra, then the following inclusions hold:
$$\Lambda-{\rm Gproj}\ \subseteq\ ^\perp \Lambda\ \subseteq \{{\rm locally\ Gorenstein\ modules }\}.$$
\end{cor}

\begin{prop}\label{propiedades de módulos ortogonales a su resol} Let $\Lambda$ be an Artin algebra and $M$ a locally Gorenstein module. Then the following statements hold:
\begin{enumerate}
\item if $M$ is non-projective then pd$(M)=\infty$,
\item each direct summand of $M$ is a locally Gorenstein module,
\item $\Omega^n(M)$ is a locally Gorenstein module, for all non-negative integer $n$.
\end{enumerate}
\end{prop}

\begin{proof}
\begin{enumerate}
\item If pd$(M)=n<\infty$ then Ext$_\Lambda^n(M, P)\neq 0$, for some $P\in \mathbf{P}_{\Lambda} M$, which is a contradiction.
\item Let $M'$ be a direct summand of $M$. The result follows from the fact that $\mathbf{P}_{\Lambda} M'\subseteq \mathbf{P}_{\Lambda} M$ and from the additivity of the bifunctor Ext$_\Lambda^i(-,-)$. 
\item It follows from $\mathbf{P}_{\Lambda} \Omega^n(M)\subseteq \mathbf{P}_{\Lambda} M$ and the isomorphisms ${\rm Ext}_\Lambda^{i}(\Omega^n(M),-)={\rm Ext}_\Lambda^{i+n}(M,-)$, $i\geq 1$. 
\end{enumerate}
\end{proof}

\begin{rem} The direct sum of locally Gorenstein modules is not necessarily a locally Gorenstein module. In fact, in Example \ref{2}.2 the modules $P_3$ and $S_4$ are locally Gorenstein modules, but $P_3\oplus S_4$ is not a locally Gorenstein module. However, if $X$ and $Y$ are locally Gorenstein modules such that $ \mathbf{P}_{\Lambda} X = \mathbf{P}_{\Lambda} Y $ then $ X \oplus Y $ is a locally Gorenstein module.   
\end{rem}

In the next session we will demonstrate that both Igusa-Todorov functions $\phi$ and $\psi$ vanish on the family of locally Gorenstein modules. Previously we need the following result. 
\begin{prop}\label{secita}  Let $\Lambda$ be an Artin algebra and $M$ a $\Lambda$-module. If $M$ is a locally Gorenstein module then $\phi(M)=\psi(M)$. 
\end{prop}
\begin{proof}
Suppose $\phi(M)=n$. Since $M$ is a locally Gorenstein module then $\Omega^n(M)$ is a locally Gorenstein module, and then any direct summand  of $\Omega^n(M)$ is a locally Gorenstein module. So each direct summand  of $\Omega^n(M)$ is projective or has an infinite projective dimension.\\
Therefore $\max\{\pd(N)\ /\ N| \Omega^n(M) \mbox{ with }\pd(N)<\infty\}=0$ and $\psi(M)=\phi(M)$.   
\end{proof}

\section{Igusa\ -\ Todorov functions vanish on locally Gorenstein modules}

Let $X$ be a locally Gorenstein $\Lambda$-module and $\mathbf{P}_{\Lambda} X=\{P_1,\,P_2,\,...,\,P_r\}$ the set of all indecomposable projective direct summands of projective modules at one minimal projective resolution of $X$. We will consider a full subcategory of mod $\Lambda$ determined by ${\rm add}((\bigoplus_{i\geq 0}\Omega^i(X))\oplus (\bigoplus_{j=1}^rP_j))$, we will denote this subcategory with $\chi_{_X}$. We will also denote with $\underline{\chi}_X$, or simply $\underline{\chi}$, the subcategory stable by projectives in ${\rm add}(\mathbf{P}_{\Lambda} X)$, that is: 

\begin{itemize}
\item The objects of $\underline{\chi}_X$ are the objects of $\chi_{_X}$.
\item ${\rm Hom}_{\underline{\chi}_X}(M,N)=\frac{{\rm Hom}_{\chi_{_X}}(M,N)}{P_{\chi_{_X}}(M,N)}$, where $P_{\chi_{_X}}(M,N)$ is the set of morphisms from $M$ to $N$ that factor through a projective in ${\rm add}(\mathbf{P}_{\Lambda} X)$. 
\item The composition is that induced by $\chi_{_X}$.
\end{itemize}

\begin{lem} Let $\underline{f}:M\rightarrow N$ be a morphisms in \chss. Then $\underline{f}=\underline{0}$ in \chs if and only if $\underline{f}=\underline{0}$ in $\underline{\modu\,}\Lambda$. In particular, \chs is a full subcategory of $\underline{\modu}\,\Lambda$.
\end{lem}
\begin{proof}
If $\underline{f}=\underline{0}$ in \chs then $\underline{f}=\underline{0}$ in $\underline{\modu}\,\Lambda$.

\noindent Conversely, if $\underline{f}=\underline{0}$ in $\underline{\modu}\,\Lambda$ then there exist a projective module  $P$ in $\modu\,\Lambda$ and two morphisms $h:M\rightarrow P$ and $j:P\rightarrow N$ such that $f=jh$. Let's see
that $f$ factors through a projective module in \chss. Indeed, consider the following commutative diagram in $\modu\,\Lambda$
$$\xymatrix{
 & & & & M\ar[dl]^{h} \ar[dd]^{f}& \\&&&P\ar@{-->}[dl]^u\ar[rd]^{j}&& \\
0\ar[r]&\Omega(N)\ar[r]^{\iota_N}&\PN\ar[rr]^{\pi_N}&&N\ar[r]&0}$$

\noindent Since $\pi_N$ is an epimorphism and $P$ is a projective module, there is  $u:P\rightarrow \PN$ such that $j=\pi_N u$. Then $f=jh=\pi_Nuh$. Now, taking $uh:M\rightarrow \PN$ we have that $f$ factors through $\PN$, with  $M,\ N,\ \PN$ objects in $\chi_{_X}$ and  $uh$, $\pi_N $, morphisms in $\chi_{_X}$. 
\end{proof}

As a consequence of the proof of the last Lemma, we obtain:

\begin{rem}\label{obs 2.1}
Let $\underline{f}:M\rightarrow N$ be a morphisms in \chs. Then $\underline{f}=\underline{0}$ in \chs if and only if  $f:M\rightarrow N$ factors through $\PN$.
\end{rem}

\begin{lem}\label{lem A}
Let $X$ be a locally Gorenstein $\Lambda$-module. If $M\in $ \chs, then $\Em(M,P)=0$ $\,\forall \,m\geq 1$, $\forall P\in \mathbf{P}_{\Lambda} X$. 
\end{lem}
\begin{proof}
Let $M$ be in \chss, that is, $M=(\bigoplus_{i=0}^nX_i)\oplus P'$, where each $X_i$ is a finite sum of summands direct of $\Omega^i(X)$ and $P'$ is a finite sum of projective modules in $\mathbf{P}_{\Lambda} X$. From the additivity in the first variable of bifunctor Ext$_\Lambda^m$, we have that 
$$\Em(M,P)=\Em(\bigoplus_{i=0}^nX_i, P)\oplus \Em(P',P).$$

Since $P'$ is a projective module we have that $\Em(P',P)=0$ for all $m\geq 1$. Then  $$\Em(M,P)=\bigoplus_{i=0}^n\Em(X_i,P)\;\forall m\geq 1. $$
We know that $\rm{Ext}_\Lambda^{m+i}(X,P)=0$ for all $m+i\geq 1$ and for all $P\in \mathbf{P}_{\Lambda} X$. Then for $m\geq 1$ we have that
$$\displaystyle \Em(X_i,P)\ | \bigoplus_{finite}\Em(\Omega^i(X),P)=\bigoplus_{finite}{\rm Ext}_\Lambda^{m+i}(X,P)=0.$$
Therefore $\Em(M,P)=0$ for all $m\geq 1$ and for all $P$ en $\mathbf{P}_{\Lambda} X$.
\end{proof}

\noindent Observe that, in the proof of the last Lemma we have demonstrated that $M\in $ \chs, is also locally Gorenstein, if $X$ is.

\noindent Now we  presents the central result of this work.

\begin{theorem}
The functor $\underline{\Omega}|_{\underline{\chi}_X}$ from $\underline{\chi}_X$ to $\underline{\chi}_X$, is full and faithful.
\end{theorem}

\begin{proof}

\noindent For simplicity, in this proof we will write  $\Omega$ instead of $\underline{\Omega}|_{\underline{ \chi}_X}$. 

\noindent We start showing that $\Omega$ is well defined on the objects, i.e., $\Omega(M)$ is an object in \chs, for all $M$ in \chss. Let  $M$ be an object in \chss, i.e., $M=(\bigoplus_{i=0}^nX_i)\oplus P$ where each $X_i$  is a finite sum of summands of $\Omega^i(X)$, and $P$ is a finite sum of projectives of $\mathbf{P}_{\Lambda} X=\{P_1,...,P_r\}$. We have that $\Omega (M)=\Omega(\bigoplus_{i=0}^nX_l)\oplus\Omega(P)=\bigoplus_{i=0}^n\Omega(X_i)\oplus 0$, where each $\Omega(X_i)$ is a finite sum of summands of $\Omega^{i+1}(X)$, so  $\Omega(M)$ is an object of  \chss. Therefore $\Omega({\rm Obj}(\underline{\chi}_X))\subseteq {\rm Obj}(\underline{\chi}_X)$.
 
 \vspace{0.5cm}

\noindent Now, lets see that $\Omega$ is well defined on morphism, i.e., if $\underline{f}:M\rightarrow N $ is a morphisms in \chs such that $\underline{f}=\underline{0}$, then $\Omega(\underline{f})=\underline{0}$ in \chss, where $\Omega(\underline{f})=\underline{\Omega(f)}$.

\noindent Let $\underline{f}:M\rightarrow N $ be a zero morphism in \chs, that is $\underline{f}=\underline{0}$ in \chss. By Remark \ref{obs 2.1}, $f$ factors through the projective cover  $ \pi_N: \PN \rightarrow N$, i.e., there exists $t:M\rightarrow \PN$ such that $\pi_N t=f$. 

\noindent Therefore we have the following diagram, where the squares and the lower triangle are commutative 
$$\xymatrix{
0\ar[r]&\Omega(M)\ar[r]^{\iota_M}\ar[d]^{\Omega(f)}&
\PM\ar[r]^{\pi_M}\ar[d]^{p}&M\ar@{-->}[dl]^t\ar[r]\ar[d]^{f}&0\\
0\ar[r]&\Omega(N)\ar[r]^{\iota_N}&\PN\ar[r]^{\pi_N}&N\ar[r]&0}$$

\noindent Since $f\pi_M=\pi_Np$, we have
$$0=\pi_Np-f\pi_M=\pi_Np-\pi_Nt\pi_M=\pi_N(p-t\pi_M).$$

\noindent Then, by the universal property of the kernel, there exists $\alpha:\PM\rightarrow \Omega(N)$ such that $\iota_N\alpha=p-t\pi_M$.

\noindent Now we have the following diagram
$$\xymatrix{
0\ar[r]&\Omega(M)\ar[r]^{\iota_M}\ar[d]^{\Omega(f)}&
\PM\ar[r]^{\pi_M}\ar[d]^{p}\ar@{-->}[dl]^{\alpha}&M\ar@{-->}[dl]^t\ar[r]\ar[d]^{f}&0\\
0\ar[r]&\Omega(N)\ar[r]^{\iota_N}&\PN\ar[r]^{\pi_N}&N\ar[r]&0}$$

\noindent and $\iota_N\alpha\iota_M=(p-t\pi_M)\iota_M=p\iota_M-t\pi_M\iota_M=p\iota_M$.
On the other hand, using that $p\iota_M=\iota_N\Omega(f)$ we get that 
$ \iota_N\alpha\iota_M=\iota_N\Omega(f)$. Then $\alpha\iota_M=\Omega(f)$, since $\iota_N$ is a monomorphism.

\noindent We have now the following commutative diagram 
$$\xymatrix{\Omega(M)\ar[dd]^{\Omega(f)}\ar[dr]^{\iota_M}\\&P_0(M)\ar[dl]^{\alpha}&\mbox{ with }P_0(M)\mbox{ in }\underline{\chi}_X.\\\Omega(N)}$$

\noindent  Therefore $\Omega(f)$ factors through $P_0(M)$ and then we have that $\underline{\Omega(f)}=\underline{0}$ in \chs.

\vspace{0.5cm}

\noindent Finally we prove that the functor $\Omega:\underline{\chi}_X \rightarrow \underline{\chi}_X$ is  faithful and full.

\noindent Let $M,\,N\in$ \chs\ and consider $\Omega:\underline{\chi}_X(M,N)\rightarrow\underline{\chi}_X(\Omega(M),\Omega(N))$. We can show that there exists

\noindent $H:\underline{\chi}_X(\Omega(M),\Omega(N))\rightarrow\underline{\chi}_X(M,N)$ such that 
$$H\Omega|_{\underline{\chi}_X(M,N)}= Id_{\underline{\chi}_X(M,N)} \mbox{ and } 
\Omega|_{\underline{\chi}_X(M,N)}H=Id_{\underline{\chi}_X (\Omega(M), \Omega(N))}.$$

As a matter of fact, to define $H$, consider the following commutative diagram:
$$\def\objectstyle{\scriptstyle}
\def\labelstyle{\scriptstyle}
\SelectTips{cm}{}
\vcenter{\xymatrix{&0\ar[d]&0\ar[d]&0\ar[d]&\\
0\ar[r]&\Ho(M,\Omega(N))\ar[r]^{\pi_M^*}\ar[d]^{\iota_{N_*}}&\Ho(\PM,\Omega(N))\ar[r]^{\iota_M^*}\ar[d]^{\iota_{N_*}}&\Ho(\Omega(M),\Omega(N))\ar[r]\ar[d]^{\iota_{N_*}}&\E(M,\Omega(N))
\\0\ar[r]&\Ho(M,\PN)\ar[r]^{\pi_M^*}\ar[d]^{\pi_{N_*}}&\Ho(\PM,\PN)\ar[r]^{\iota_M^*}\ar[d]^{\pi_{N_*}}&\Ho(\Omega(M),\PN)\ar[r]^{\hspace{0.1cm} \mbox{ \tiny{Lemma \ref{lem A}}} }\ar[d]^{\pi_{N_*}}& 0=\E(M,\PN)&(I)
\\0\ar[r]&\Ho(M,N)\ar[r]^{\pi_M^*}\ar[d]&\Ho(\PM,N)\ar[r]^{\iota_M^*}\ar[d]&\Ho(\Omega(M),N)\ar[r]\ar[d]&\E(M,N)
\\&\E(M,\Omega(N))&0&\E(\Omega(M),\Omega(N))&}}$$

\noindent obtained from the exact sequences of the following commutative diagram:
$$\xymatrix{
0\ar[r]&\Omega(M)\ar[r]^{\iota_M}\ar[d]^{\Omega(f)}&
\PM\ar[r]^{\pi_M}\ar[d]^{p}&M\ar[r]\ar[d]^{f}&0\\
0\ar[r]&\Omega(N)\ar[r]^{\iota_N}&\PN\ar[r]^{\pi_N}&N\ar[r]&0}$$

\noindent Define $H:\underline{\chi}(\Omega(M),\Omega(N))\rightarrow\underline{\chi}(M,N)$ in the following way:

\noindent Given $\underline{g}:\Omega(M)\rightarrow\Omega(N)$ in \chss, let $g:\Omega(M)\rightarrow\Omega(N)$ be a representative class of $\underline{g}$. From diagram $(I)$ it follows that $\inc(g)=\iota_Ng\in \Ho(\Omega(M),\PN)$, and since 
$\imct$ is an epimorphism we know that there exits $p:\PM\rightarrow \PN$ such that
$$\imct(p)=\inc(g).$$
Since $\pnc\inc=0$ we have $0=\pnc\inc(g)=\pnc\imct(p)$. So, by the commutativity of the second lower square of the diagram $(I)$,  we know that $\imct\pnc(p)=0$. Then $\pnc(p)\in Ker(\imct)=Im(\pmct)$ and  there exists $h\in\Ho(M,N)$ such that $\pmct(h)=\pnc(p)$.

\noindent At this point, we have the following commutative diagram:
$$\xymatrix{
0\ar[r]&\Omega(M)\ar[r]^{\iota_M}\ar[d]^{g}&
\PM\ar[r]^{\pi_M}\ar@{-->}[d]^p&M\ar[r]\ar@{-->}[d]^{h}&0\\
0\ar[r]&\Omega(N)\ar[r]^{\iota_N}&\PN\ar[r]^{\pi_N}&N\ar[r]&0}$$

\vspace{3mm}

\noindent We define
$$H:\underline{\chi}(\Omega(M),\Omega(N))\rightarrow\underline{\chi}(M,N)
\mbox{ with } 
H(\underline{g})=\underline{h}$$
\begin{center}where $\pmct(h)=\pnc(p)$ with $p$ such that $\imct(p)=\inc(g)$.\end{center}
\vspace{.5cm}

\noindent Let's see that $H$ is well defined, i.e,
$H$ does not depend on the choice of $p$  and  if $\underline{g}=\underline{0}$ in \chs then $\underline{h}=\underline{0}$ in \chss.

\noindent Let $p,p'\in \Ho(\PM,\PN)$ such that $\imct(p)=\inc(g)$, $\pnc(p)=\pmct(h)$, $\imct(p')=\inc(g)$ and $\pnc(p')=\pmct(h')$ with $h,h'\in\Ho(M,N)$. Since $\imct(p)=\imct(p')$ we have that $\imct(p-p')=0$ and $p-p'\in Ker(\imct)=Im(\pmct)$. So, there is $\alpha:M\rightarrow \PN$ such that $\pmct(\alpha)=p-p'$. By the commutativity of the left lower square of the diagram $(I)$, we have that
$\pmct\pnc(\alpha)=\pnc\pmct(\alpha)=\pnc(p-p')=\pnc(p)-\pnc(p')=\pmct(h)-\pmct(h')=\pmct(h-h')$. Now, knowing that $\pmct$ is a monomorphism, we get that $\pnc(\alpha)=h-h'$.

\noindent Therefore, we have the following commutative diagram 
$$\xymatrix{M\ar[dd]^{h-h'}\ar[dr]^{\alpha}\\&P_0(N)\ar[dl]^{\pi_N}\\N}$$

\noindent where $h-h'$ factors through  $P_0(N)$, with $P_0(N)$  in $\underline{\chi}_X$. By Remark  \ref{obs 2.1}, $\underline{h-h'}=\underline{0}$ in \chss, i.e., $\underline{h}=\underline{h'}$ in \chss, so $H:\underline{\chi}(\Omega(M),\Omega(N))\rightarrow\underline{\chi}(M,N)$ does not depend on the choice of $p$.

\vspace{3mm}

\noindent Now we will show that if $H(\underline{g})=\underline{h}$  with $\underline{g}=\underline{0}$ in \chs, then $\underline{h}=\underline{0}$ in \chss.
 
\noindent By Remark  \ref{obs 2.1}, since $\underline{g}=\underline{0}$ in \chss, $g$ factors through  $P_0(\Omega(N))$, we have the following commutative diagram  
$$\xymatrix{
&0\ar[r]&\Omega(M)\ar[dl]^{t}\ar[r]^{\iota_M}\ar[dd]^{g}&
\PM\ar[r]^{\pi_M}\ar[dd]^p&M\ar[r]\ar[dd]^{h}&0\\
&P_0(\Omega(N))\ar[dr]^{\epsilon}&&&&\\
&0\ar[r]&\Omega(N)\ar[r]^{\iota_N}&\PN\ar[r]^{\pi_N}&N\ar[r]&0}$$

\noindent where $g=\epsilon t$. 

\noindent Considering the pushout of the morphisms  $\iota_M$ and $t$, $(PO,\iota_M,t)$, we have the following commutative diagram
$$\xymatrix{
&0\ar[r]&\Omega(M)\ar[r]^{\iota_M}\ar[d]^{t}&
\PM\ar[r]^{\pi_M}\ar[d]^r&M\ar[r]\ar@2{-}[d]&0&\hspace{1cm}(II)\\
\delta: &0\ar[r]&P_0(\Omega(N))\ar[r]^{s}&PO\ar[r]&M\ar[r]&0} 
$$

\noindent Since $P_0(\Omega(N))$ and $M$ are objects in  \chss, we have that ${\rm Ext}_\Lambda^m(M,P_0(\Omega(N))=0$ for all $m\geq 1$, by Lemma \ref{lem A}, in particular ${\rm Ext}_\Lambda^1(M,P_0(\Omega(N))=0$ and then the short exact  sequence $\delta$ splits. So there exists a morphism $z':PO\rightarrow P_0(\Omega(N))$ such that $z's=id_{P_0(\Omega(N))}$. 

\noindent Taking $z=z'r:\PM\rightarrow P_0(\Omega(N))$ we have the following diagram
$$\xymatrix{
&0\ar[r]&\Omega(M)\ar[dl]^{t}\ar[r]^{\iota_M}\ar[dd]^{g}&
\PM\ar[lld]^z\ar[r]^{\pi_M}\ar[dd]^p&M\ar[r]\ar[dd]^{h}&0&\hspace{1cm}(III)\\
&P_0(\Omega(N))\ar[dr]^{\epsilon}&&&&\\
&0\ar[r]&\Omega(N)\ar[r]^{\iota_N}&\PN\ar[r]^{\pi_N}&N\ar[r]&0} $$

\noindent Let's show that $g=\epsilon z \iota_M$. By diagram $(II)$, we have
$st=r\iota_M$. Then $z'st=z'r\iota_M$. From $z's=id_{P_0(\Omega(N))}$ and $z'r\iota_M=z\iota_M$, it follows that $t=z\iota_M$.
By diagram $(III)$,  $g=\epsilon t=\epsilon z\iota_M\mbox{, wich means that,  }g=\imct(\epsilon z)$.
\noindent So, we have the following diagram with the commutative squares and the commutative upper triangle
$$\xymatrix{
0\ar[r]&\Omega(M)\ar[r]^{\iota_M}\ar[d]^{g}&
\PM\ar[r]^{\pi_M}\ar[d]^p\ar[dl]^{\epsilon z}&M\ar[r]\ar[d]^{h}&0\\
0\ar[r]&\Omega(N)\ar[r]^{\iota_N}&\PN\ar[r]^{\pi_N}&N\ar[r]&0}$$

\noindent From $p\iota_M=\iota_Ng=\iota_N\epsilon z\iota_M$ it follows that $(p-\iota_N\epsilon z)\iota_M=0$, i.e., $\imct(p-\iota_N\epsilon z)=0$. Then $p-\iota_N\epsilon z\in Ker(\imct)=Im(\pmct)$.
Since the middle row of the diagram $(I)$ is exact, there is a morphism  $\beta:M\rightarrow \PN$ such that $p-\iota_N\epsilon z=\pmct(\beta)=\beta\pi_M$. 

\noindent By the commutativity of lower left square of the diagram $(I)$, we have
$$\pmct\pnc(\beta)=\pnc\pmct(\beta)=\pnc(p-\iota_N\epsilon z)=\pnc(p)-\pnc(\iota_N\epsilon z)=\pnc(p).$$

\noindent On the other hand, by the definition of $H$, we know that $\pnc(p)=\pmct(h)$. So, we obtain 
$$\pmct\pnc(\beta)=\pnc(p)=\pmct(h).$$

\noindent And then, since $\pmct$ is a monomorphism, we have $\pnc(\beta)=h$, i.e, $\pi_N\beta=h$. 
Therefore $h$ factors through $P_0(N) $ and $\underline{h}=\underline{0}$ in \chss.

\vspace{0.5cm}

\noindent It only remains to prove that the compositions of the maps $H:\underline{\chi}(\Omega(M),\Omega(N))\rightarrow\underline{\chi}(M,N)$ and 
\linebreak$\Omega:\underline{\chi}(M,N)\rightarrow\underline{\chi}(\Omega(M),\Omega(N))$
 are the corresponding identities.

\vspace{3mm}

\noindent First, using the diagram $(I)$, let's show that $H\Omega=id_{\underline{\chi}(M,N)}$, i.e, $H\Omega(\underline{f})=\underline{f}$ for all $\underline{f}$ in $\underline{\chi}(M,N)$. If $f\in\Ho(M,N)$ is a representative of $\underline{f}$, we know that $\pmct(f)= f\pi_M\in\Ho(\PM,N)$. Since $\pnc$ is an epimorphism, there is $p':\PM\rightarrow \PN$ such that $\pnc(p')=\pmct(f)$. Also, for $p'$ we have  $\pnc(p'\iota_M)=\pnc\imct(p')=\imct\pnc(p')=\imct(\pmct(f))=0$. So, $p'\iota_M\in Ker(\pnc)=Im(\inc)$, from where there exists  $g\in\Ho(\Omega(M),\Omega(N))$ such that $\inc(g)=\imct(p')$. Then, by the definition of $H$, we have that $H(\underline{g})=\underline{f}$.

\noindent Considering the commutative diagram
$$\xymatrix{
0\ar[r]&\Omega(M)\ar[r]^{\iota_M}\dar[d]^{\hspace{1.5mm} g}_{\Omega(f)}&
\PM\ar[r]^{\pi_M}\dar[d]^{\hspace{1.5mm} p'}_{p}&M\ar[r]\ar[d]^{f}&0\\
0\ar[r]&\Omega(N)\ar[r]^{\iota_N}&\PN\ar[r]^{\pi_N}&N\ar[r]&0}$$

\noindent we obtain, by the construction of $g$ and the definition of $\Omega(f)$, that $g-\Omega(f)$ factors through a projective module $P$ in $\modu\,\Lambda$. Even more, by Remark \ref{obs 2.1}, $g-\Omega(f)$ factors through  $P_0(\Omega(N))$ and $\underline{g-\Omega(f)}=\underline{0}$ in \chss. Then $\underline{g}=\underline{\Omega(f)}$ in \chs and $H\Omega(\underline{f})=H(\underline{\Omega(f)})=H(\underline{g})=\underline{f}$.

\vspace{3mm}

\noindent To show that  $\Omega H=id_{\underline{\chi}(\Omega(M),\Omega(N))}$, let's consider $\underline{g}$ in $\underline{\chi}(\Omega(M),\Omega(N))$ and let  $\underline{h}=H(\underline{g})$, with $\underline{h}$ in $\underline{\chi}(M,N)$. By the definition of $H$, there exists $p:P_0(M)\rightarrow P_0(N)$ such that the following diagram is commutative
$$\xymatrix{0\ar[r]&\Omega(M)\ar[r]^{\iota_M}\ar[d]^g&P_0(M)\ar[r]^{\pi_M}\ar[d]^p&M\ar[r]\ar[d]^h&0\\0\ar[r]&\Omega(N)\ar[r]^{\iota_N}&P_0(N)\ar[r]^{\pi_N}&N\ar[r]&0}$$

\noindent where $g\in {\rm Hom}_\Lambda(\Omega(M),\Omega(N))$ and $h\in{\rm Hom}_\Lambda(M,N)$ are representatives of  $\underline{g}$ y $\underline{h}$, respectively. Since the definition of $\underline{\Omega}:\underline{\modu}\,\Lambda\rightarrow \underline{\modu}\,\Lambda$ does not depend on the lift considered, let´s calculate  $\underline{\Omega(h)}$ taking the last mentioned morphism  $p:P_0(M)\rightarrow P_0(N)$. Then, since the above diagram is commutative, we have that $\underline{\Omega(h)}=\underline{g}$ in $\underline{\modu}\,\Lambda$. So, $\underline{\Omega(h)}=\underline{g}$ in \chss. Finally, $\Omega(H(\underline{g}))=\Omega(\underline{h})=\underline{g}$.

\end{proof}

\begin{cor}\label{cor teo principal}
Under the conditions of the previous theorem $H:\underline{\chi}(\Omega(M),\Omega(N))\rightarrow\underline{\chi}(M,N)$ and 

\noindent$\Omega:\underline{\chi}(M,N)\rightarrow\underline{\chi}(\Omega(M),\Omega(N))$ are inverse maps. 
\end{cor}

\begin{theorem}\label{teo1}
Let $X$ be a locally Gorenstein module. If $M$ belongs to \chss, then $\phi(M)=0$.
\end{theorem}
\begin{proof}
Let $[\underline{\chi}_X]$ be the subgroup of $K_0$ generated by the classes of indecomposable modules in \chss. 
Let's probe that $\On|_{[\underline{\chi}_X]}$ is injective. Let $[X_1]$ and $[X_2]$ be in [\chss] such that $\On([X_1])=\On([X_2])$. From the definition of $\On$ we have that $[\Omega(X_1)]=[\Omega(X_2)]$ and from Proposition 3.1 of \cite{FLM} we have that $\Omega(X_1)\cong\Omega(X_2)$ in $\underline{\modu}\, \Lambda$ and therefore $\Omega(X_1)\cong\Omega(X_2)$ in \chss. Then there are morphisms $\underline{f}:\Omega(X_1)\rightarrow\Omega(X_2)$ and $\underline{g}:\Omega(X_2)\rightarrow\Omega(X_1)$  such that $\underline{g}\underline{f}=\underline{id}_{\Omega(X_1)}$ and $\underline{f}\underline{g}=\underline{id}_{\Omega(X_2)}$. 
 
\noindent  From Corollary \ref{cor teo principal}, $ \Omega $ and $ H $ are inverse maps, we have that $ \underline{g} \underline{f} = \Omega (H (\underline {g}))\Omega(H( \underline{f}))=\Omega (H (\underline {g})H( \underline{f})) $. So, from $ \Omega (\underline{id}_{X_1}) = \underline{id}_{\Omega{(X_1)}} = \underline{g} \underline{f} $ we have that $ \Omega (\underline{id}_{X_1}) = \underline{g} \underline{f} = \Omega (H (\underline{g}) H (\underline{f})) $. Therefore, since $\Omega$ is a faithful functor we have that $\underline{id}_{X_1}=H(\underline{g})H(\underline{f})$. Analogously, $\underline{id}_{X_2}=H(\underline{f})H(\underline{g})$. Then $X_1\cong X_2$ in \chs and  therefore $X_1\cong X_2$ in \mas. Now, from \cite{FLM} we get $[X_1]=[X_2]$ in $K_0$ and therefore $[X_1]=[X_2]$ in [\chss].  

\noindent Finally, $\On|_{[\underline{\chi}_X]}$ is injective and $\phi(M)=0$ for all $M$ en \chss.

\end{proof}

\begin{cor}\label{phi de ortogonal cero}
Let $X$ be a locally Gorenstein module. If $M$ is in \chss, then $\psi(M)=0$.
\end{cor}

As a corollary, we obtain further proof of the following result, previously demonstrated in \cite{LM}.

\begin{cor}{\cite{LM}}
If $M\in\,  ^\perp \Lambda$, that is, $M$ is a semi-Gorenstein projective, then $\phi(M)=\psi(M)=0$. In particular, if  $M$ is a Gorenstein projective $\Lambda$-module, then $\phi(M)=\psi(M)=0$.
\end{cor}

\section{ The $n$-locally Gorenstein algebras }

The homological dimensions defined with the Igusa-Todorov functions are related with the finitistic dimension and global dimension. Indeed, the following inequalities are satisfied: 

$${\rm fin.dim}(\Lambda)\leq\phid(\Lambda)\leq \psid(\Lambda)\leq {\rm gl.dim}(\Lambda).$$
\subsection{Definition, examples and firsts results}
\noindent In this section we define the $n$-locally Gorenstein algebras, and we calculate the finitistic dimension, the $\phi$-dimension and the $\psi$-dimension for this family of  algebras. Finally, we exhibit examples of families of $n$-locally Gorenstein algebras.

\begin{defn}\label{def algebras l-G}
An Artin algebra $\Lambda$ will be called   
{\bf n-locally Gorenstein} if each module $X$ in $\Omega^n({\rm mod}\,\Lambda)$ is locally Gorenstein, that is, for all $X\in\Omega^n(\modu \Lambda)$ we have that $\Em(X,P)=0$ for all $m\geq 1$ and for all $P\in\mathbf{P}_{\Lambda} X$.
\end{defn}

\begin{ejes}\label{ejemplo de alg n-ort}
\noindent \begin{enumerate}
\item Selfinjective algebras are $0$-locally Gorenstein. 

\item The $n$-Gorenstein algebras are $n$-locally Gorenstein algebras, because $\Omega^n(\modu \Lambda)\subseteq {\cal G}_{\Lambda}$.

\item 
In fact, the hypothesis id$({_\Lambda}\Lambda)=n<\infty$, 
is sufficient to show that $\Lambda$ is an $n$-locally Gorenstein algebra. This follows from Example \ref{2}\ (b), and also from the known fact that findim$(\Lambda) \leq n$.

\item A one-point extension $\Lambda [M]$ of a $n$-Gorenstein algebra $\Lambda$, is a $(n+1)$-locally Gorenstein algebra. For this purpose observe that $\Omega (\modu \Lambda [M])\subseteq \modu \Lambda$.

\end{enumerate}
\end{ejes}

\begin{eje}
Consider the Nakayama algebra $\Lambda=kQ/I$ given by the following quiver:
$$\xymatrix{Q\!\!\!\! & \!\!\!\!  = \!\!\!\! & \!\!\!\! 1\ar@/^/[r]^\alpha&2\ar@/^/[l]^\beta}$$
and by the ideal $I=\langle \alpha \beta \alpha\rangle$. Then $\Lambda$ is not a Gorenstein algebra but it's a 2-locally Gorenstein algebra.
\end{eje}  

\noindent Let $\Lambda$ be a $n$-locally Gorenstein algebra. Then, by definition,  $\Omega^n(M)$ is a locally Gorenstein module, for any $M\in \modu\Lambda$. This is the key to being able to prove that all $n$-locally Gorenstein algebras verify the finitistic dimension conjecture. Although the result also arises as a corollary of the forthcoming Theorem \ref{teo phi en alg. n-ort}, because ${\rm fin.dim}(\Lambda)\leq\phid(\Lambda)$,
we give a direct proof of it, as it involves the famous finitistic dimension conjecture.

\begin{prop}
The finitistic dimension of a $n$-locally Gorenstein algebra $\Lambda$ is finite: $$fin.dim (\Lambda)\leq n.$$ 
\end{prop}  

\begin{proof}
Let $M$ be a $\Lambda$-module of finite projective dimension. Then, $\pd (\Omega^n (M))< \infty$ and so $\Omega^n(M)$ is projective by Proposition \ref{propiedades de módulos ortogonales a su resol}.1. Therefore, $\pd (M)\leq n$ and fin.dim$(\Lambda )\leq n<\infty$.
\end{proof}

\noindent Recently, in \cite{BM} and \cite{HK}, examples of algebras with finite finitistic dimension and infinite $\phid$- dimension were exhibited. However, if $\Lambda$ is a n-locally Gorenstein algebra then both dimensions are finite, as the following theorem shows.

\begin{theorem}\label{teo phi en alg. n-ort}
Let $\Lambda$ be a $n$-locally Gorenstein algebra. Then $$\phid(\Lambda)\leq n<\infty.$$
\end{theorem}
\begin{proof}
If $M\in\modu\Lambda$ then $\Omega^n(M)$ is a locally Gorenstein module by Proposition \ref{propiedades de módulos ortogonales a su resol}.3..  From Theorem \ref{teo1} we have that $\phi(\Omega^n(M))=0$, and from Proposition 1.4.7 \cite{HL}
 $$\phi(M)\leq \phi(\Omega^n(M))+n=0+n=n,$$
\noindent and then $\phid(\Lambda)\leq n<\infty$.
\end{proof}

\begin{cor}
Let $\Lambda$ be a $n$-locally Gorenstein algebra. Then $$\phid(\Lambda)=\psid(\Lambda)\leq n<\infty.$$
\end{cor}
\begin{proof} If $M\in\modu\Lambda$ then $\Omega^n(M)$ is a locally Gorenstein module, then: 
 $$\psi(M)\leq \psi(\Omega^n(M))+n = \phi(\Omega^n(M)+n=0+n,$$ where the equalities hold by Proposition \ref{secita} and Theorem \ref{teo1}. Finally, from definition, $\phi(X)\leq \psi(X)$, for any $X\in\modu\Lambda$, and the result is done. 
\end{proof}

\subsection{Morita context algebras that are locally Gorenstein}\label{seccion 5.2}
Given $A$ and $B$ two Artin locally Gorenstein algebras, the goal of this section is to show that the algebra 

\noindent $\Lambda_{(0,0)}=\begin{pmatrix} A & _AN_B \\
_BM_A & B\end{pmatrix}$, given by a Morita context, is also locally Gorenstein, when certain conditions are imposed on the bimodules $_BM_A$ and $_AN_B$.
       
\begin{prop}\label{cubiertas proy en Lambda}
    Let be $\Lambda_{(0,0)}=\begin{pmatrix}
         A & _AN_B \\ _BM_A & B \end{pmatrix}$ 
        an algebra  arising from a Morita context, $X$ be an  $A$-module and $Y$ be a B-module. Then:
        \begin{enumerate}
        \item If $P_0(X)$ is the projective cover of $X$ in $\modu A$ 
        then $(P_0(X),M\otimes_A P_0(X),Id_{M\otimes P_0(X)},0)$ is the projective cover of $(X,0,0,0)$ in $\modu \Lambda_{(0,0)}$.
        \item If  $Q_0(Y)$ is the projective cover of $Y$ in $\modu B$ 
        then $(N\otimes_B Q_0(Y), Q_0(Y),0,Id_{N\otimes Q_0(Y)})$ is the projective cover of $(0,Y,0,0)$ in  $\modu \Lambda_{(0,0)}$.  
        \end{enumerate}    
\end{prop}
\begin{proof} We show  only 1. because 2. is analogous. Let 
$p:P_0(X)\rightarrow X$ be the projective cover of $X$ on 
$\modu A$. 

    It is clear that $(p,0): (P_0(X),M\otimes_A P_0(X),Id_{M\otimes P_0(X)},0)\rightarrow (X,0,0,0)$ is an epimorphism of $\Lambda_{(0,0)}$-modules. We need to show that $(p,0)$ is a minimal morphism: let $(a',b'):(X',Y',f,g)\rightarrow (P_0(X),M\otimes P_0(X),Id_{M\otimes P_0(X)},0)$ be a $\Lambda_{(0,0)}$-morphism such that $(p,0)\circ (a',b')$ is an epimorphism. We must show that $(a',b')$  is an epimorphism. On the one hand, we remark that $a'$ is an epimorphism since $p$ is a minimal morphism.

On the other hand, from the commutative diagram 
$$\xymatrix{M\otimes_A X'\ar[d]_{Id_M\otimes a'}\ar[r]^f & Y' \ar[d]^{b'}\\M\otimes_A P_0(X) \ar[r]_{Id_{M\otimes P_0(X)}}& M\otimes_A P_0(X)}$$
we have that $b'\circ f=Id_{M\otimes P_0(X)}\circ (Id_M\otimes a')$, and since $Id_{M\otimes P_0(X)}\circ (Id_M\otimes a')$ is an epimorphism (because we proved above that $a'$ is an epimorphism) we have that $b'$ it is also an epimorphism.       
\end{proof}

\begin{prop}\label{Resol. proy y sizigias utiles}
   Let $A$ and $B$ be Artin algebras and $\Lambda_{(0,0)}= \begin{pmatrix}
         A & _AN_B \\ _BM_A & B \end{pmatrix}$ be the algebra associated with the Morita context. Let $X$ an $A$-module such that  $M\otimes_A P=0$ for all $P\in \mathcal{P}_A X$ and  $Y$ be an  $B$-module such that $N\otimes_B Q=0$ for all $Q\in \mathcal{P}_B Y$. Then: 
   
\begin{enumerate}
    \item If $\xymatrix{...\ar[r]&P_n\ar[r]&...\ar[r]&P_1\ar[r]&P_0\ar[r]&X\ar[r]&0}$ is the minimal projective resolution of $X$ in $\modu A$ then 
    $$\xymatrix{...\ar[r]&(P_n,0,0,0)\ar[r]&...\ar[r]&(P_1,0,0,0)\ar[r]&(P_0,0,0,0)\ar[r]&(X,0,0,0)\ar[r]&0}$$ is the minimal projective resolution of $(X,0,0,0)$ in  $\modu \Lambda_{(0,0)}$.
    \item $\Omega^n_{\Lambda}((X,0,0,0))=(\Omega^n_A(X),0,0,0)$ for all $n\geq 0$.
    \item If $\xymatrix{...\ar[r]&Q_n\ar[r]&...\ar[r]&Q_1\ar[r]&Q_0\ar[r]&Y\ar[r]&0}$ is the minimal projective resolution of $Y$ in $\modu B$ then 
    $$\xymatrix{...\ar[r]&(0,Q_n,0,0)\ar[r]&...\ar[r]&(0,Q_1,0,0)\ar[r]&(0,Q_0,0,0)\ar[r]&(0,Y,0,0)\ar[r]&0}$$ is the minimal projective resolution of $(0,Y,0,0)$ in $\modu\Lambda_{(0,0)}$.
    \item $\Omega_{\Lambda}^n((0,Y,0,0))=(0,\Omega_B^n(Y),0,0)$ for all $n\geq 0$.
\end{enumerate}
   
\end{prop}
\begin{proof}
By Proposition \ref{cubiertas proy en Lambda} 1. and $M\otimes_A P=0$ for all $P\in\mathcal{P}_A X$, we have that $((P_0(\Omega_A^nX)),0,0,0),(p_n,0))$ is the projective cover of $(\Omega^n_A(X),0,0,0)$  for all $n\geq 0$, so $1$ and $2$ hold.

A similar argument proves $3$ and $4$.
     
\end{proof}

\begin{cor}
Let $A$ and $B$ be Artin algebras and $\Lambda_{(0,0)}= \begin{pmatrix}
         A & _AN_B \\ _BM_A & B \end{pmatrix}$ be the algebra associated with the Morita context. Let $X$ an $A$-module such that  $M\otimes_A P=0$ for all $P\in \mathcal{P}_A X$ and  let $Y$ be an  $B$-module such that $N\otimes_B Q=0$ for all $Q\in \mathcal{P}_B Y$. Then every projective $\Lambda_{(0,0)}$-module in $ \mathcal{P}_{\Lambda} (X,Y,0,0)$ has the form $(P,Q,0,0)$ with $P\in\mathcal{P}_A X$ and $Q\in\mathcal{P}_B Y$.

\end{cor}

\noindent If we assume that $A$ is a $n$-locally Gorenstein algebra and that $B$ is a $m$-locally Gorenstein algebra, then we want to show that $\Lambda_{(0,0)}=\begin{pmatrix}A& _AN_B \\ _BM_A & B \end{pmatrix}$, under certain conditions on the bimodules $M$ and $N$, is a $(l+1)$-locally Gorenstein algebra, with $l=max\{m, n\}$.

\begin{lem}\label{lema 5.9}
Let $A$ and $B$ $n$ and $m$ locally Gorenstein algebras, respectively, and let $_BM_A$ and $_AN_B$ projective bimodules such that  $M\otimes_A N=0$ and $N\otimes_BM=0$. We consider $\Lambda_{(0,0)}$ the algebra associated with the Morita context. If for each  $X\in \Omega^n(\modu A)$ we have that $M\otimes_A P=0$ for all $ P\in\mathcal{P}_A X$ and for each $Y\in  \Omega^m(\modu B)$
we have that $N\otimes_B Q=0$ for all $Q\in\mathcal{P}_B Y$, then for each $(X,Y,f,g)\in \Omega^t(\modu \Lambda_{(0,0)})$ where $t=$max$\{n,m\}$, the following
statements hold:
\begin{center}
$\Omega_\Lambda (X,Y,f,g)=(\Omega_A(\overline{X}),0,0,0)\oplus(0,\Omega_B(\overline{Y}),0,0),$ where $\overline{X}\in \Omega_A^t(\modu A)$ and $\overline{Y}\in \Omega_B^t(\modu B)$.
\end{center}

\end{lem}

\begin{proof}

\noindent We observe that if $t=0$ then $A$ and $B$ are $0$-locally Gorenstein algebras. So, the conditions: for each  $X\in \Omega^n(\modu A)$ we have that $M\otimes_A P=0$ for all $ P\in\mathcal{P}_A X$ and for each $Y\in  \Omega^m(\modu B)$
we have that $N\otimes_B Q=0$ for all $Q\in\mathcal{P}_B Y$, imply that $M=0$ and $N=0$. Then, $\Lambda_{(0,0)}=A \times B$ is a  product of algebras $A$ and $B$. The result follows directly.\\

\noindent We assume $t\geq 1$, let $(X,Y,f,g)\in \Omega^t(\modu\Lambda_{(0,0)})$ then there is $(X', Y', f', g')\in \modu\Lambda_{(0,0)}$ such that $$(X,Y,f,g)=\Omega^t_{\Lambda}((X', Y', f', g')).$$

  \noindent Then we have that:
   $$\begin{array}{cl}
        (X,Y,f,g)&= \Omega^{t-1}_{\Lambda}(\Omega_{\Lambda}(X',Y',f', g')) \\
        &=\Omega^{t-1}_{\Lambda}((\Omega_A (X'),M\otimes_A P_0^{X'},1\otimes i_0,0)\oplus (N\otimes_B Q_0^{Y'},\Omega_B (Y'),0,1\otimes j_0))\\
        
        & =\Omega^{t-2}_{\Lambda}(\Omega_{\Lambda}((\Omega_A (X'),M\otimes_A P_0^{X'},1\otimes i_0,0))\oplus \Omega_{\Lambda}((N\otimes_B Q_0^{Y'},\Omega_B (Y'),0,1\otimes j_0)))\\
        & =\Omega_{\Lambda}^{t-2}(\Omega_{\Lambda}((\Omega_A (X'),0,0,0))\oplus \Omega_{\Lambda}((0,\Omega_B (Y'),0,0)))\\
        & =\Omega_{\Lambda}^{t-1}((\Omega_A (X'),0,0,0)\oplus (0,\Omega_B (Y'),0,0)) (*).
   \end{array}$$
\noindent where the second and fourth equalities are satisfied by Proposition 3.3 \cite{MD} and by (2) and (3) of Remark 2.12 of \cite{LV}, respectively. Finally,  by (4) and (5) of Remark 2.12  of \cite{LV} we have that
$$\begin{array}{cl}
       (X,Y,f,g) & 
       =\Omega_\Lambda((\Omega_A^{t-1}(X'),0,0,0)\oplus (0,\Omega_B^{t-1} (Y'),0,0)),
   \end{array}$$
   \noindent so
$$\begin{array}{cl}
       (X,Y,f,g) & 
       =\Omega_\Lambda((\Omega_A^{t-1}(X'),0,0,0))\oplus \Omega_\Lambda((0,\Omega_B^{t-1} (Y'),0,0))\\
       &=(\Omega_A^{t} (X'),M\otimes_A P_0^{\Omega_A^{t-1}(X')},1\otimes i_0,0)\oplus (N\otimes_B P_0^{\Omega_B^{t-1}(Y')},\Omega_B^t (Y'),0,1\otimes j_0)
 \end{array}$$
\noindent where the second equality is satisfied by Proposition 3.3 of \cite{MD}. Note that from the last equality we have that  $\Omega^t_A(X')$ is a direct summand of $X$ and $\Omega^t_A(Y')$ is a direct summand of $Y$.
   
\noindent Then, applying $\Omega_\Lambda$ to the last equality in $(*)$, we have that 
$$\begin{array}{cl}
       \Omega_\Lambda(X,Y,f,g) & 
       =\Omega_\Lambda^t((\Omega_A(X'),0,0,0)\oplus (0,\Omega_B (Y'),0,0))\\
       &=\Omega_\Lambda((\Omega_A^{t}(X'),0,0,0)\oplus (0,\Omega_B^{t} (Y'),0,0))\\
       &=(\Omega_A^{t+1}(X'),0,0,0)\oplus (0,\Omega_B^{t+1} (Y'),0,0))
   \end{array}$$
\noindent where the second  equality is satisfied by Remark 2.12 of \cite{LV}, and the third equality is satisfied by Proposition \ref{Resol. proy y sizigias utiles} since $\Omega^t_A(X')$ is a direct summand of $X$ and $\Omega^t_A(Y')$ is a direct summand of $Y$. 

\noindent Taking $\overline{X}=\Omega^t_A(X')$ e $\overline{Y}=\Omega^t_B(Y')$, we have what we desired.
\end{proof}

\noindent Note that in the previous proof, for the case $t=0$,  have that $\Lambda_{(0,0)}=A\times B$ is a $0$-locally Gorenstein algebra (and therefore $1$-locally Gorenstein algebra).

\begin{theorem}\label{Teo4}
  Let $A$ and $B$ $n$ and $m$ locally Gorenstein algebras, respectively, and let $_BM_A$ and $_AN_B$ projective bimodules such that  $M\otimes_A N=0$ and $N\otimes_BM=0$. Let $\Lambda_{(0,0)}=\begin{pmatrix}A& _AN_B \\ _BM_A & B \end{pmatrix}$, be the algebra associated with the Morita context. If for each  $X\in \Omega^n(\modu A)$ we have that $M\otimes_A P=0$ for all $ P\in\mathcal{P}_A X$ and for each $Y\in  \Omega^m(\modu B)$ we have that $N\otimes_B Q=0$ for all $Q\in\mathcal{P}_B Y$, then:
  \begin{enumerate}
      \item[(1)] 
      all the modules in  $\Omega^{t+1}_{\Lambda}(\modu \Lambda)$ are locally Gorenstein $\Lambda$-modules, with $t=max\{m,n\}$;
      \item[(2)]
$\Lambda_{(0,0)}$ is a $t+1$ locally Gorenstein algebra.
        \end{enumerate} 
\end{theorem}

\begin{proof}
\begin{enumerate}
    \item[(1)]
Let $t=max\{m,n\}$ and let $(X',Y',f',g')$ be a $\Lambda$-module in $\Omega^{t+1}(\modu \Lambda)$. Then, there is $(X,Y,f,g)\in \Omega^t(\modu \Lambda)$ such that $\Omega_{\Lambda}((X,Y,f,g))=(X',Y',f',g')$.

\noindent By the previous Lemma we have that:

$$(X',Y',f',g')=\Omega_\Lambda (X,Y,f,g)=(\Omega_A(\overline{X}),0,0,0)\oplus(0,\Omega_B(\overline{Y}),0,0) $$
where $\overline{X}\in \Omega_A^{t}(\modu A)$ and $\overline{Y}\in \Omega_B^{t}(\modu B)$, and therefore each one of them is locally Gorenstein, as $A$-module, $B$-module, respectively. This implies that $X'= \Omega_A(\overline{X});\ Y'= \Omega_B(\overline{Y})$ and:

$$ \mathcal{P}_{\Lambda} (X',Y', f',g') = \mathcal{P}_{\Lambda} (X',Y',0,0)=\mathcal{P}_{\Lambda} (\Omega_A(\overline{X}),0,0,0) \cup \mathcal{P}_{\Lambda} (0,\Omega_B(\overline{Y}),0,0).$$

 So, we have to show that Ext$^i_\Lambda((X',Y',0,0),(P,Q,0,0))=0$ for all  $i\geq 1$ and for all $(P,Q,0,0)\in \mathcal{P}_{\Lambda} (X',Y',0,0)$.

From the projective resolutions:
$$\xymatrix{\cdots\ar[r]&P_2\ar[r]&P_1\ar[r]&P_0\ar[r]&\Omega_A (\overline{X})\ar[r]&0}\ \mbox{ in }\modu A,$$
and
$$\xymatrix{\cdots\ar[r]&Q_2\ar[r]&Q_1\ar[r]&Q_0\ar[r]&\Omega_B(\overline{Y})\ar[r]&0}\ \mbox{ in }\modu B,$$
we obtain, since $M\otimes_A P=0$,  $\forall P\in\mathcal{P}_A X$ and  $N\otimes_B Q=0$, $\forall Q\in\mathcal{P}_B Y$, the following exact sequences:
$$\xymatrix{\cdots\ar[r]&T_A(P_2)\ar[r]&T_A(P_1)\ar[r]&T_A(P_0)\ar[r]&
 (\Omega_A (\overline{X}),0,0,0) \ar[r]&0}$$
and
$$ \xymatrix{\cdots\ar[r] & T_B(Q_2)\ar[r] & T_B(Q_1) \ar[r] & T_B(Q_0) \ar[r] &(0,\Omega_B(\overline{Y}),0,0) \ar[r]&0}.$$

\noindent Then, since the pairs of functors  $(T_A,U_A)$ and $(T_B,U_B)$ are pairs of adjoint functors we have:

$$\scriptsize
\xymatrix@C=2mm@R=5mm{\Hom_\Lambda((\Omega_A(\overline{X}),0,0,0),(P,0,0,0))\ar[d]\ar[r]&\Hom_\Lambda(T_A(P_0),(P,0,0,0))\ar[d]^{\simeq}\ar[r]&\Hom_\Lambda(T_A(P_1),(P,0,0,0))\ar[d]^{\simeq}\ar[r]&\Hom_\Lambda(T_A(P_2),(P,0,0,0))\ar[d]^{\simeq}\cdots&\\
\Hom_A(\Omega_A(\overline{X}),P)\ar[r]&\Hom_A(P_0,P)\ar[r]&\Hom_A(P_1,P)\ar[r]&\Hom_A(P_2,P)\ \ \cdots&}$$

and

$$\scriptsize
\xymatrix@C=2mm@R=5mm{\Hom_\Lambda(0,(\Omega_B(\overline{Y}),0,0),(0,Q,0,0))\ar[d]\ar[r]&\Hom_\Lambda(T_B(Q_0),(0,Q,0,0))\ar[d]^{\simeq}\ar[r]&\Hom_\Lambda(T_B(Q_1),(0,Q,0,0))\ar[d]^{\simeq}\ar[r]&\Hom_\Lambda(T_B(Q_2),(0,Q,0,0))\ar[d]^{\simeq}\cdots&\\
\Hom_B(\Omega_B(\overline{Y}),Q)\ar[r]&\Hom_B(Q_0,Q)\ar[r]&\Hom_B(Q_1,Q)\ar[r]&\Hom_B(Q_2,Q)\ \ \cdots&}$$

where from

$$\Ext^{i}_\Lambda((\Omega_A(\overline{X}),0,0,0), (P,0,0,0))\simeq \Ext^{i}_A(\Omega_A(\overline{X}),P)$$
and
$$\Ext^{i}_\Lambda(0,(\Omega_B(\overline{Y}),0,0), (0,Q,0,0))\simeq \Ext^{i}_B(\Omega_B(\overline{Y}),Q).$$ 

Therefore, for $(P,Q,0,0)\in \mathcal{P}_{\Lambda} (X',Y',0,0)$, we have that:
$$\Ext^i_\Lambda((X',Y',f',g'),(P,Q,0,0))=\Ext^{i}_\Lambda(\Omega_{\Lambda}((X,Y,f,g)),(P,Q,0,0))=$$
$$= \Ext^{i}_\Lambda((\Omega_A(\overline{X}),0,0,0)\oplus(0,\Omega_B(\overline{Y}),0,0)),(P,Q,0,0))=$$ $$=\Ext^{i}_\Lambda((\Omega_A(\overline{X}),0,0,0),(P,Q,0,0))\oplus \Ext^{i}_\Lambda((0,\Omega_B(\overline{Y}),0,0),(P,Q,0,0))\simeq$$
$$\Ext^{i}_A(\Omega_A(\overline{X}),P)\oplus \Ext^{i}_B(\Omega_B(\overline{Y}),Q)=0,\, \forall i \geq 1.$$

\noindent The last equality follows from the fact that $\Omega_A(\overline{X})$ and $\Omega_B(\overline{Y})$ are locally Gorenstein modules, as $A$-module, $B$-module, respectively,  and that $P\in \mathcal{P}_A X'\subseteq \mathcal{P}_A \Omega_A (\overline{X}) $ and $Q\in \mathcal{P}_B Y'\subseteq \mathcal{P}_B \Omega_B(\overline{Y})$,  for what was shown before.

\item[(2)]
Is a direct consequence of the first item.
    
\end{enumerate}
\end{proof}
\subsection{Examples constructed via Morita contexts}

We present here four examples, constructed via Morita contexts, two respecting the hypotheses of Theorem \ref{Teo4} and two in which the hypotheses do not hold. In one of these examples, we take non-projective bimodules, and in the second on, the tensors between the bimodules are not zero. However, the second of these examples is also a locally Gorenstein algebra.\\

\noindent The following example verifies the hypotheses of Theorem \ref{Teo4}, generating a new locally Gorenstein algebra from locally Gorenstein algebras, using the Morita context.

\begin{eje}
Let $Q$ be the quiver:
$$\xymatrix{5\ar@(dl,ul)&4\ar[l]& 3\ar[l]& 2 \ar[l] & 1 \ar[l] }$$

\noindent We consider the algebra $\Lambda_{(0,0)}=\begin{pmatrix}
        A & N \\  M  & B \end{pmatrix}$ arising from a Morita context, talking $A=B=KQ/F^2$ and the bimodules  $M=Be_6\otimes e_3A$ and  $N=Ae_1\otimes e_8B$. Observe that $A$ and $B$ are 4-locally Gorenstein algebras. 
        
        The associated quiver to $\Lambda_{(0,0)}$ is:

$$\xymatrix{ &&5\ar@(dl,ul) &4\ar[l]& 3\ar[l] \ar@(d,u)[d]& 2 \ar[l] & 1 \ar[l] \\
\ar@(dl,ul) 10&9\ar[l]& 8\ar[l] \ar@(d,d)[urrrr]& 7 \ar[l] & 6 \ar[l]   & & }.$$

\vspace{0.5 cm}
\noindent Observe that there is no new relations.\\

\noindent The algebra $\Lambda_{(0,0)}$ is a $4$-locally Gorenstein algebra, just like  $A$ and $B$.\\
\end{eje}

Next we give an example of a locally Gorenstein infinite type representation algebra.

\begin{eje}\label{5.11} We consider a double point extension of the famous selfinjective algebra given by Liu and Schulz, \cite{R}, obtaining an Artin algebra given by the following quiver:
$$\xymatrix{Q:&  &3\ar@(ul,dl)_{\alpha_3}\ar@(dl,dr)_{\alpha_4}\ar@(ur,ul)_{\alpha_5}&2\ar[l]^{\alpha_2}&1\ar[l]^{\alpha_1}& }$$
 and the relations: $I=\{\alpha_2\alpha_1,\ \alpha_3^2,\ \alpha_4^2,\ \alpha_5^2,\ \alpha_4\alpha_3+r\alpha_3\alpha_4,\ \alpha_5\alpha_4+r\alpha_3\alpha_5,\ \alpha_3\alpha_5+r\alpha_5\alpha_3\}$, where $r\in k$ is a non zero element of an algebraically closed field $k$.

 Now we consider the Morita context algebra defined by: $\Lambda_{(0,0)}=\begin{pmatrix}
        A &  N \\  M  & B \end{pmatrix}$, taking $A=B=kQ/I$ and the bimodules $_BM_A$ and  $_AN_B$ such that the quiver of  $\Lambda_{(0,0)}$ is:
       $$\xymatrix{ &3\ar@(ul,dl)_{\alpha_3}\ar@(dl,dr)_{\alpha_4}\ar@(ur,ul)_{\alpha_5}&2\ar[l]^{\alpha_2}&1\ar[l]^{\alpha_1}\ar[ddl]_{\beta_1}&\\
        \\
        &3'\ar@(ul,dl)_{\alpha'_3}\ar@(dl,dr)_{\alpha'_4}\ar@(ur,ul)_{\alpha'_5}&2'\ar[l]^{\alpha'_2}&1'\ar[l]^{\alpha'_1}\ar[uul]_{\beta_2}&}$$ 
 with no new relations. By Theorem \ref{Teo4} we get a locally Gorenstein infinite representation type  algebra.\\
 \end{eje}

Next is an example where the bimodules $_BM_A$ and $_AN_B$ are not projective. We show that the algebra  $\Lambda_{(0,0)}$,
obtained is not locally Gorenstein.

\begin{eje}

\noindent Let $Q$ be the quiver:
$$\xymatrix{Q:&3 \ar@(dl,ul) & 2 \ar[l] & 1 \ar[l] }$$

\noindent We consider the algebra $\Lambda_{(0,0)}=\begin{pmatrix}
        A &  N \\  M  & B \end{pmatrix}$ arising from a Morita context, taking $A=KQ/F^2$, $B=K$ and the bimodules $_BM_A$ and  $_AN_B$, in such a way that the quiver associated with $\Lambda_{(0,0)}$ is:

$$\xymatrix{  \ar@(d,u)[dr]_{\beta_1} 3 \ar@(dl,ul)^{\alpha_3} & 2 \ar[l]^{\alpha_2} & 1 \ar[l]^{\alpha_1} \\&
  4  \ar@(r,d)[ur]_{\beta_2}  & }$$
and the  set of relations is: $\{\alpha_2\alpha_1,\  \alpha_3\alpha_2,\ \alpha_3^2,\ \beta_1\alpha_2,\  \beta_2\beta_1,\ \alpha_1\beta_2 \}$.\\

\noindent First observe that above construction implies both $_BM_A$ and $_AN_B$ are not projective bimodules.  The indecomposable projective  $\Lambda_{(0,0)}$-modules and indecomposable injective $\Lambda_{(0,0)}$-modules can be represented as:
$$\xymatrix{P_1:\txt{1\\2} &P_2:\txt{2\\3}& P_3:\txt{3\\3 4 \\4 \hspace{0.2 cm} }&P_4:\txt{4\\1}&I_1:\txt{4\\1}&I_2:\txt{1\\2}&I_3:\txt{3 2\\\hspace{1mm}3}&I_4:\txt{3\\3\\4}}$$
\noindent The minimal projective resolution of $S_3$ is:
$$\xymatrix@C=2mm{\cdots&\ar[dr]P_3\oplus P_3\ar[rr]&&\ar[dr]P_3\oplus P_2\ar[rr]&&\ar[dr]P_3\oplus P_1\ar[rr]&&P_3\oplus P_4\ar[rr]\ar[dr]&&P_3\ar[r]&S_3\ar[r]&0\\\ &&\ar[ur]\txt{3\\4}\oplus 3&&\ar[ur]\txt{3\\4}\oplus 2&&\ar[ur]\txt{3\\4}\oplus 1&&\txt{3\\4}\oplus 4\ar[ur]&&&}$$
\noindent Then $S_3$ is a direct summand of $\Omega^{4k}(S_3)$ for all $k\geq 0$, so $S_3$ its in $\Omega^{4k}$(mod $\Lambda_{(0,0)}$) for all $k\geq 0$.
 
\noindent On the other hand, since  Ext$^1_{\Lambda_{(0,0)}}(S_3,P_3)\neq 0$ we have that Ext$^{4k+1}_{\Lambda_{(0,0)}}(S_3,P_3)\neq 0$ for all $t\geq 0$, and so $S_3$ is not  locally Gorenstein module, because $\mathcal{P}_{S_3}=\{P_1,P_2,P_3,P_4\}$.\\
\noindent Therefore $\Lambda_{(0,0)}$ is not  locally Gorenstein algebra.\\
\end{eje} 

 \begin{eje} Let's consider again the initial algebra given in the Example \ref{5.11}, that is the algebra given by the quiver:
$$\xymatrix{Q:& &3\ar@(ul,dl)_{\alpha_3}\ar@(dl,dr)_{\alpha_4}\ar@(ur,ul)_{\alpha_5}&2\ar[l]^{\alpha_2}&1\ar[l]^{\alpha_1}&}$$
 and the set of relations $I=\{\alpha_2\alpha_1,\  \alpha_3^2,\ \alpha_4^2,\ \alpha_5^2,\ \alpha_4\alpha_3+r\alpha_3\alpha_4,\ \alpha_5\alpha_4+r\alpha_3\alpha_5,\ \alpha_3\alpha_5+r\alpha_5\alpha_3\}$, where  $r\in k$ is a non zero element of an algebraically closed field $k$.

 Now we consider the Morita context algebra defined by: $\Lambda_{(0,0)}=\begin{pmatrix}
        A &  N \\  M  & B \end{pmatrix}$, taking $A=B=kQ/I$ and the  bimodules $_BM_A$ and  $_AN_B$ such that the quiver is given by:
       $$\xymatrix{ & 3\ar@(ul,dl)_{\alpha_3}\ar@(dl,dr)_{\alpha_4}\ar@(ur,ul)_{\alpha_5} & 2\ar[l]^{\alpha_2} & 1\ar[l]^{\alpha_1}\ar@/^/[dd]^{\beta_1} & \\
        \\
        & 3'\ar@(ul,dl)_{\alpha_3'}\ar@(dl,dr)_{\alpha_4'}\ar@(ur,ul)_{\alpha_5'} & 2'\ar[l]^{\alpha_2'} & 1'\ar[l]^{\alpha_1'}\ar@/^/[uu]^{\beta_2} & }$$ 
 and the set of relations is given by: $$I'=\{\alpha_2\alpha_1,\ \alpha_3^2,\ \alpha_4^2,\ \alpha_5^2,\ \alpha_4\alpha_3+r\alpha_3\alpha_4,\ \alpha_5\alpha_4+r\alpha_3\alpha_5,\ \alpha_3\alpha_5+r\alpha_5\alpha_3,\  \beta_1\beta_2,\ \beta_2\beta_1, $$
$$ \alpha_2'\alpha_1',\ \alpha_3'^2,\ \alpha_4'^2, \ \alpha_5'^2,\ \alpha_4'\alpha_3'+r\alpha_3'\alpha_4',\ \alpha_5'\alpha_4+r'\alpha_3'\alpha_5',\ \alpha_3'\alpha_5'+r\alpha_5'\alpha_3'\}.$$

 \noindent Although this example does not satisfy the hypotheses of Theorem \ref{Teo4}, namely, $M\otimes_A N\neq 0\neq N\otimes_B M$, it can be shown using Example \ref{ejemplo de alg n-ort}\ \textit{(3)} that the algebra $\Lambda_{(0,0)}$ is 3-locally Gorenstein.
 
 \end{eje}

\section*{Acknowledgements}
Author L.A. is grateful to Asociación de Universidades Grupo Montevideo for funding a one-month research stay at Instituto de Matemática y Estadística “Rafael Laguardia”, Universidad de la República, which enabled the development of this study.

\bibliographystyle{plain}
\bibliography{Referencias}

\end{document}